\documentclass[preprint]{elsarticle}

\usepackage[english]{babel}
\usepackage{amsmath,amsfonts,amssymb,color,amsthm}

\usepackage{natbib}

\usepackage{url}

\usepackage{booktabs}
\usepackage[ruled,vlined,algo2e]{algorithm2e}
\usepackage{mathdots}
\usepackage{float}
\usepackage{physics}
\usepackage{enumitem}
\numberwithin{equation}{section}
\newcommand{\R}{\mathbb R}

\newcommand{\C}{\mathbb C}
\def\eps{\varepsilon}

 \newtheorem{thm}{Theorem} 
 \newtheorem{lem}{Lemma}
  \newdefinition{rmk}{Remark}

\newtheorem{definition}{Definition}
\newtheorem{problem}{Problem}
\newtheorem{corollary}{Corollary}
\newtheorem{example}{Example}
\numberwithin{thm}{section}
\numberwithin{lem}{section}
\numberwithin{rmk}{section}
\numberwithin{definition}{section}
\numberwithin{problem}{section}
\numberwithin{corollary}{section}
\numberwithin{example}{section}

\usepackage{multirow,bigdelim}
\usepackage[detect-none]{siunitx}
\usepackage[caption=false]{subfig}

\usepackage{booktabs}

\usepackage{float}
\usepackage{mathtools}

\begin{document}
\title{Generalized algorithms for the  approximate matrix polynomial GCD of reducing data uncertainties with application to MIMO system and control}

 \author[1,2]{Antonio Fazzi\corref{cor1}}
 \ead{Antonio.Fazzi@vub.be}
\author[1]{Nicola Guglielmi} \ead{nicola.guglielmi@gssi.it}
\author[2]{Ivan Markovsky} \ead{imarkovs@vub.ac.be}
\cortext[cor1]{Corresponding author}
\address[1]{Gran Sasso Science Institute (GSSI), Viale F. Crispi 7, 67100 L'Aquila, Italy}
\address[2]{Vrije Universiteit Brussel (VUB), Department ELEC, Pleinlaan 2, 1050 Brussels, Belgium}
\begin{abstract}
Computation of (approximate) polynomials common factors is an important problem in several fields of science, like control theory  and signal processing. While  the problem has been widely studied for scalar polynomials, the scientific literature in the framework of matrix polynomials seems to be  limited to the problem of exact greatest common divisor computation. In this paper, we  generalize two algorithms from scalar to matrix polynomials.  The first  one is fast and simple.  The second one is more accurate but computationally  more expensive. We test the performances of the two algorithms and observe similar behavior  to the one in the scalar case.  Finally we describe an application to multi-input multi-output linear time-invariant dynamical  systems.
\end{abstract}


\maketitle

\section{Introduction}
Polynomials common factors  computation is an important problem in several scientific fields due to its applications \cite{Markovsky2018}. In this paper we deal with common factors for matrix polynomials, which are matrices whose elements are polynomials, or equivalently polynomials with matrix coefficients.  Readers not familiar with matrix polynomials can refer for example  to \cite{Gohberg2009,Rosenwasser2006}.

The computation of a Greatest Common Divisor (GCD) $C(\lambda)$  of two  matrix polynomials $A(\lambda)$ and $B(\lambda)$ appears in several problems  in multivariable control  \cite{Emre1980,ForneyJr.1975,Basilio1997}.
The problem has been studied by many authors and through different techniques. Some authors find the GCD as a  combination of polynomials \cite{MacDuffee1933} or transform the block matrix $[A(\lambda) \  B(\lambda)]$ into $[C(\lambda) \   \  0]$ \cite{Wolovich1974}. 
Other methods use  the generalized Sylvester matrix  \cite{anderson1976, Bitmead1978}. 

The most popular  references  study the properties of the resultant for matrix polynomials, e.g.  \cite{anderson1976, Gohberg2010, Gohberg2008, Kaashoek2010}, or they deal with exact common factor  computations  for matrix  polynomials  \cite{Basilio1997, Bitmead1978, Moness1982}.  
Anyway in some applications (see Section \ref{sec:applications}) it is needed to  compute  approximate common factors, due to measurement noise or other 
perturbations  on the data.

 The Approximate GCD problem has been extensively studied for scalar polynomials; but  in the framework of multivariable control systems we deal with matrix polynomials and, up to our knowledge, there is no algorithm for solving the problem in the matrix case. The goal of this paper is to generalize the algorithms proposed in \cite{Qiu1997} and in   \cite{Guglielmi2017, Fazzi2018} from scalar to matrix polynomials. 
 \textcolor{blue}{
 This paper is organized as follows:   Section \ref{sec:exactgcd}   relates to the exact GCD computation in the matrix case, and the properties of the generalized resultant;   Section \ref{sec:ss} generalizes the subspace method   (Algorithm \ref{alg:sub}) of  \cite{Qiu1997},  while  Section \ref{sec:ode}   generalizes the ODE-based method  (Algorithm \ref{alg:ODE}) of  \cite{Guglielmi2017, Fazzi2018}.   Section \ref{sec:numericaltest} shows the performance of the algorithms. Finally applications in the framework of linear time-invariant systems are considered  in Section \ref{sec:applications}.}

\paragraph{Notations}
\begin{itemize}
\item $A(\lambda)$, $B(\lambda)$ are two (square) coprime matrix  polynomials, $\hat{A}(\lambda), \hat{B}(\lambda)$ are  perturbations of $A(\lambda), B(\lambda)$ having a common factor (the outputs of the proposed algorithms).    They can be factored as  $\hat{A} = C \bar{A}$, $\hat{B} = C \bar{B}$; $C$ denotes the (monic) common factor;

\item $m$ is the dimension of the matrices $A, B$, $n$ is the degree of the polynomials (we assume they have the same degree), $d$ is the degree of the sought common factor;


\item $S_{\ell}$ denotes a structured Sylvester matrix whose dimensions depend on the parameter $\ell$ (see Section \ref{sec:sylvres}); $A \in \mathcal{S}$ means that the matrix $A$ has the Sylvester structure   and $P_{\mathcal{S}}(\cdot)$ is the  operator which orthogonally  project the argument onto the set $\mathcal{S}$;


\item we denote by  $\| \cdot \|_F$ the  Frobenius norm of a matrix induced by  the  Frobenius inner product $\langle A, B \rangle = \tr(A^{\top} B)$;

\item $\tau(C)$ denotes the Toeplitz matrix built from the coefficients of the matrix polynomial $C(z)$;

\item a dot on a function denotes its time  derivative (we deal with univariate functions only).
\end{itemize}
We restrict in the following to the case of two matrix polynomials and 
we assume  both the matrices $A, B$ to be square in order to simplify the notation. Anyway the proposed algorithms work  if one of the two matrices is rectangular
(as pointed out in Remark \ref{rmk:lefttorightswitch} we need only two matrices  having the same number of rows or columns) 
 and they could be extended to more than two polynomials.
 Throughout the paper we use without distinction the terms GCD and common factors. 

\section{Matrix polynomial GCD approximation}
\label{sec:exactgcd}
We analyze in this section how to approach the  common factors computation in the case of matrix polynomials, emphasizing the main differences with respect to the scalar case. 
The first difference arising when we consider matrices instead of scalars is the loss of commutativity. Henceforth, we need to distinguish between  right and left divisors. In the following  we  focus on left divisors but right  divisors have obvious counterparts. 

\begin{definition}[Left divisor of two matrix polynomials]
\label{def:gcld}
A  (exact) common left divisor  of two matrix polynomials  $A(\lambda)$ and $B(\lambda)$, having the same number of rows, is any matrix polynomial  $C(\lambda)$ such that
\begin{equation*}
\label{eq:rightdivisor}
A(\lambda) = C(\lambda) \bar{A}(\lambda)  \ \ \  \ \ B(\lambda) = C(\lambda) \bar{B}(\lambda) 
\end{equation*}
for some matrix  polynomials  $\bar{A}(\lambda)$, $\bar{B}(\lambda)$;
%
\end{definition} 

\begin{rmk}
\label{rmk:lefttorightswitch}
The definition of left (right) divisor is meaningful only in the case the two matrices have the same number of rows (columns). 
If we transpose the matrix polynomials, we can switch between left and right common factors.
\end{rmk}

In the framework of scalar polynomials, two common factors (or, in general, two polynomials) are equivalent up to a constant factor. A similar property holds  true in the matrix case: two matrix polynomials are equivalent up to multiplication with unimodular matrices.
\begin{definition}[Unimodular matrix polynomials]
	\label{def:unimodular}
Let $U(\lambda)$ be a square   matrix polynomial of dimension $m$. Then $U(\lambda)$ is  a unimodular matrix polynomial if there exists a $m \times m$  matrix polynomial  $V(\lambda)$ such that $V(\lambda) U(\lambda) = I$. Equivalently, if $\textrm{det} (U(\lambda))$ is a non-zero constant.
\end{definition}

\begin{definition}[Matrix polynomials equivalence]
\label{def:eqmat}
 Given two matrix polynomials $C_1(\lambda)$ and $C_2(\lambda)$, they are equivalent if and only if there exist  unimodular matrix polynomials $U(\lambda)$, $V(\lambda)$ such that $C_1(\lambda) = U(\lambda) C_2(\lambda) V(\lambda)$.
\end{definition}

The following statement is helpful to understand if a given matrix polynomial is unimodular: 
	$U(\lambda)$ is a unimodular matrix polynomial if and only if it is associated with  a finite sequence of   the following transformations:

\begin{enumerate}
\item interchange two columns: it is equivalent to  the multiplication with a permutation matrix;
\item multiply a column by a nonzero constant: it is equivalent to multiplication  with a constant diagonal matrix;
\item replace the i-th column $c_i(\lambda)$ by $c_i(\lambda) + \lambda^d c_j(\lambda)$: this is equivalent to the multiplication with a matrix polynomial  equal to the identity except for the presence of $\lambda^d$ in the position $(j, i)$;
\item all the previous transformations can be applied to the rows  and they correspond to a premultiplication with a  suitable unimodular matrix. 
\end{enumerate}

\begin{rmk}
	\label{rmk:fullrank}
The set of equivalent common factors, according to Definition \ref{def:eqmat} and the last statement, is big and sometimes it can be difficult to understand if two given matrix polynomials are equivalent even for small dimensions. In order to make this problem milder  we  restrict,  in the following, to the case of monic  common factors (there is some loss of generality since we restrict to the polynomials whose leading coefficient is full rank).
This assumption is not fundamental, though;  by removing it we can compute approximate common factors of given degree whose leading coefficient is not full rank.  
\end{rmk} 


\subsection{Sylvester matrices for matrix polynomials}
\label{sec:sylvres}
Let $A$ and $B$ be $m \times m$ matrix polynomials of degree $n$. Thus
\begin{equation*}
\label{pola}
A(\lambda ) = A_0 + A_1 \lambda + \cdots + A_n \lambda^n \ \text{with} \ A_n \ne 0,
\end{equation*}
\begin{equation*}
\label{polb}
B(\lambda) = B_0 + B_1 \lambda + \cdots + B_n \lambda^n \ \text{with} \ B_n \ne 0.
\end{equation*}
We assume $n > 0$, and that the leading matrix  coefficients $A_n$ and $B_n$ are invertible, so the determinants of $A(\lambda)$ and $B(\lambda)$  are not zero. 

A useful tool in testing polynomials coprimeness is the Sylvester resultant: its straightforward generalization to the matrix case is 
  the following $2mn \times 2mn$ structured  matrix
\begin{equation*}
\label{eq:resultant}
S(A, B) =\begin{pmatrix}
A_n & \cdots & \cdots & A_0 &  &  &  \\
       & A_n     & \cdots & \cdots & A_0 &  &  \\
      &            & \ddots &             &        & \ddots & \\
      &          &              & A_n    & \cdots & \cdots & A_0 \\
 B_n & \cdots & \cdots & B_0 &  &  &  \\
       & B_n     & \cdots & \cdots & B_0 &  &  \\
      &            & \ddots &             &        & \ddots & \\
      &          &              & B_n    & \cdots & \cdots & B_0 \\
\end{pmatrix}
\begin{array}{clc}
 & \rdelim\} {4} {4mm} & \\
 &                               & n \\
 &                          &  \\
 &                           & \\
  & \rdelim\} {4} {4mm} & \\
 &                               & n \\
 &                          &  \\
 &                           & \\
 \end{array}.
\end{equation*}
 In \cite{Gohberg2010} it has been shown that 
the key property for the classical Sylvester resultant does not carry over for matrix polynomials, in particular
\begin{equation}
\label{kergeqeigs}
\text{dim} \   \ker ( S(A, B)) \geq \nu (A,B),
\end{equation}
where $\nu(A, B)$ denotes the total common multiplicity of the common eigenvalues of $A$ and $B$.  Example \ref{ex:1} shows that the inequality \eqref{kergeqeigs} can be strict.
\textcolor{blue}{
\begin{example}
	\label{ex:1}
	Let the two $2 \times 2$ matrix polynomials of degree $1$, 
\begin{equation}
\label{ex:ineqholdstrue}
 A(\lambda) = \begin{pmatrix}
-1 + \lambda & 0 \\ 1 & -1 + \lambda
\end{pmatrix}, \ \ \ B(\lambda) = \begin{pmatrix}
 \lambda & 1 \\ 0 &  \lambda - 2
\end{pmatrix}. 
\end{equation} 
We deduce easily that $A(\lambda) = A_0 + A_1 \lambda$ and $B(\lambda) = B_0 + B_1 \lambda$ where
\begin{equation*}
A_0 = \begin{pmatrix}
-1 &0\\1&-1
\end{pmatrix}, \ A_1 = \begin{pmatrix}
1&0\\0&1
\end{pmatrix}, \ B_0 = \begin{pmatrix}
0&1\\0&-2
\end{pmatrix}, \ B_1=\begin{pmatrix}
1&0\\0&1
\end{pmatrix}.
\end{equation*}
We have
\begin{equation}
  \mathbf{S}(A, B) \begin{pmatrix}
1 \\ -2 \\ 1 \\ -1
\end{pmatrix} = \begin{pmatrix}
A_1 & A_0 \\ B_1 & B_0
\end{pmatrix}  \begin{pmatrix}
1 \\ -2 \\ 1 \\ -1
\end{pmatrix} =  \begin{pmatrix}
1 & 0 & -1 & 0 \\ 0 & 1& 1 &-1 \\ 1&0&0& 1 \\ 0 & 1 & 0 & -2 
\end{pmatrix} \begin{pmatrix}
1 \\ -2 \\ 1 \\ -1
\end{pmatrix} = \begin{pmatrix} 0\\0\\0\\0 \end{pmatrix},
\end{equation}
so the kernel of the resultant has dimension (at least) $1$, but $\textrm{det}(A(\lambda))$ and $\textrm{det}(B(\lambda))$ have no common zeros, hence the matrices have no common eigenvalues. 
\end{example}}
On the other hand, given $A(\lambda), B(\lambda)$ and $\lambda_0 \in \C$, if there exists  a vector $x_0 \neq 0$ such that  $A(\lambda_0) x_0 = 0$ and $B(\lambda_0) x_0 = 0$ then det$(A(\lambda_0))=0$ and det$(B(\lambda_0)) = 0$; but the contrary is not true. Consequently, the common factors are not associated only with   the common roots of the determinants of the matrix polynomials.

In order to get the equality in \eqref{kergeqeigs} we can consider a bigger Sylvester matrix \cite{Gohberg2010}. Defining the following  resultant
\begin{equation}
\label{eq:modifiedres}
S_{\ell}(A, B) = \begin{pmatrix}
A_n & \cdots & \cdots & A_0 &  &  &  \\
       & A_n     & \cdots & \cdots & A_0 &  &  \\
      &            & \ddots &             &        & \ddots & \\
      &          &              & A_n    & \cdots & \cdots & A_0 \\
 B_n & \cdots & \cdots & B_0 &  &  &  \\
       & B_n     & \cdots & \cdots & B_0 &  &  \\
      &            & \ddots &             &        & \ddots & \\
      &          &              & B_n    & \cdots & \cdots & B_0 \\
\end{pmatrix}
\begin{array}{l}
  \\[-10mm] \rdelim\}{4}{4mm}[$\ell-n$] \\  \\  \\[6mm]  \rdelim\}{4}{4mm}[$\ell-n$] \\ \\
\end{array}
\end{equation}
we have the equality in \eqref{kergeqeigs} if $\ell \geq n(m+1)$;  in the following we set $\ell = n(m+1)$.
\textcolor{blue}{
\begin{example}
	Using \eqref{ex:ineqholdstrue}, we deduce that
	\begin{equation*}
	S_3(A, B) = \begin{pmatrix}
	A_1 & A_0 & \\   & A_1 & A_0\\ B_1 & B_0 & \\   & B_1 &B_0
	\end{pmatrix} =  \begin{pmatrix}
1 & 0 & -1 & 0 & 0&0 \\
0 & 1 & 1 & -1 & 0 & 0 \\
0 & 0 & 1 & 0 & -1 & 0 \\
0 & 0 & 0 & 1 & 1 & -1 \\
1 & 0 & 0 & 1& 0 & 0  \\
0 & 1 & 0 & -2 & 0 & 0 \\
0 & 0 & 1 & 0 & 0 & 1 \\
0 & 0 & 0 & 1 & 0 & -2
\end{pmatrix}
\end{equation*}
is full rank.
\end{example}.}
\begin{rmk}
The definition of resultant in \eqref{eq:modifiedres} refers to 	right common factors. If we deal with left common factors we need to transpose it. 
\end{rmk}
\subsection{Common factor approximation}
In the past years several authors have proposed some algorithms for the computation of an exact GCD of matrix polynomials.  But in practical applications, the coefficients can be inexact due to several sources of error. 
Given coprime matrix polynomials, we are interested in computing the smallest perturbation which makes them having a common factor of given degree. 

Consider  two coprime matrix polynomials $A(\lambda)$ and $B(\lambda)$. The problem is to compute a closest pair  of matrix polynomials $\hat{A}(\lambda)$, $\hat{B}(\lambda)$ which has a non trivial  (exact) common factor of specified degree $d$. Such a common factor is called an approximate common factor for the matrices $A(\lambda)$ and B($\lambda$).
In the following we assume that the coefficient matrices are real. The distance between two pairs of matrix polynomials is defined as follows:

\begin{equation}
\label{eq:dist}
\textrm{dist}(\{A, B\}, \{\hat{A}, \hat{B}\}) = \sqrt{\sum_{j=0}^n \lVert A_j - \hat{A}_j \rVert_F^2 + \sum_{j=0}^{n} \lVert B_j - \hat{B}_j \rVert_F^2}
\end{equation}
where $A_{j}$ and $B_j$ denote the $j$-th (matrix) coefficient of the corresponding matrix  polynomial.
The formulation of the   problem   is the following:
\textcolor{blue}{
\begin{problem}{Approximate left common factor ptoblem for matrix polynomials}
\label{prob:agcd}
Given two left coprime matrix polynomials  $A=A(\lambda)$ and  $B=B(\lambda)$, a number $d \in \mathbf{N}$, compute
$$
\inf_{\substack{\{\hat{A}, \hat{B}\}: \exists C \ \text{such that} \\ \hat{A}= C \bar{A}, \hat{B} = C \bar{B} \\ C \ \text{has  degree}\  d}} \textrm{dist}(\{A, B\}, \{\hat{A}, \hat{B}\})
$$
where $A,B$ denote (with an abuse of notation) a matrix collecting the coefficients of the corresponding matrix polynomial, while  the distance is the one defined in \eqref{eq:dist}. The left  common factor $C$  is an approximate common factor for the matrix polynomials $A$ and $B$. The problem involving approximate  right common factor is analogous.  
\end{problem}}

In the following sections we propose two algorithms for solving the nonconvex optimization Problem \ref{prob:agcd} by local optimization approaches. To the best of our knowledge there is no algorithm in the  literature to compute its solution. Our proposals come from the generalization of two algorithms proposed in the scalar case: the subspace method \cite{Qiu1997} and an ODE-based algorithm \cite{Fazzi2018}. We list for each algorithm the main points and properties, and we test their performance on some numerical examples.


\section{Generalized subspace method for matrix polynomials}
\label{sec:ss}
In this section we describe how we generalize the subspace method \cite{Qiu1997} to the computation of  approximate common factors of matrix polynomials.
The original algorithm for scalar polynomials is a powerful tool  in the framework of GCD computation since it is \textit{simple to develop, easy to understand and convenient to implement}. Moreover it is one of the first algorithms capable of dealing with noise-corrupted data. However, as shown in \cite{Fazzi2018}, the performance of the subspace method can be improved  in terms of accuracy of the solution by other optimization methods.
The basic idea of the algorithm is the fact that the information on the (approximate) common factors of a set of polynomials is in the  null space of the associated resultant. 

We briefly recall how the algorithm works for scalar polynomials, as described in \cite{Qiu1997}:
\begin{enumerate}
\item Build $S$, the Sylvester matrix  of dimension $N (n+1) \times (2n+1)$ associated with the given data polynomials, where  $N$ is the number of polynomials and $n$ is the degree of the polynomials.
\item \begin{enumerate}[label=\alph*]
\item Compute 
$$ V_0=(V_{d}, \dots, V_{1})
, $$ 
the  null space of $S$  ($d$ is the degree of the sought GCD). $V_0$ has $d$ columns.
\item In order to extract the information about the GCD, reshape each column of $V_0$ into a Hankel matrix with $r=2n+1-d$ rows:
\begin{equation*}
H_i = \begin{pmatrix}
V_i (1) & V_i(2) & \cdots & V_i(d+1) \\
V_i(2) & V_i(3) & \cdots & V_i(d+2) \\
\cdots & \cdots & \cdots & \cdots \\
V_i(r) & V_i(r+1) & \cdots & V_i(2n+1)
\end{pmatrix} \ \ \ i=r+1, \dots, 2n+1.
\end{equation*}
\end{enumerate}
\item Build the matrix 
$$
R= \sum_{\substack{i=d \\  i=i-1}}^{1} H_i^T H_i
$$
 and extract the GCD by the eigenvector of $R$ corresponding to the smallest eigenvalue. The entries of such eigenvector are the coefficients of the common factor. 
\end{enumerate}

To generalize the method for matrix polynomials, we replace the scalar coefficients by matrices of dimension $m$, manipulating and reshaping the data in a suitable way. 
Similarly to the scalar case, the  algorithm works in the same way both in the computation of exact common factors or approximate common factors. This leads to high computational speed but less accurate solutions. 
 The main points of the algorithm are summarized in Algorithm \ref{alg:sub}.
 \begin{algorithm2e}[H]
 	\caption{Subspace method for the computation of (approximate) common factors of matrix polynomials}
 	\label{alg:sub}
 	\DontPrintSemicolon
 	\KwData{$A, B$ ($m \times m$ matrix polynomials of degree $n$), $d$ (degree common factor).}
 	\KwResult{$\hat{C}$ (common factor),}

 	\Begin{
 		\nl Build the structured Sylvester matrix $S_{n(m+1)}$ as in \eqref{eq:modifiedres} \;
 		\nl Compute the null space of $S_{n(m+1)}$
 		$$
 		V_0=(V_{md}, \dots, V_1),
 		$$  \;
 		\nl \For{i=md, \dots, 1}{
 		Reshape each vector of $V_0$ into a matrix with $m$ rows 
 		\begin{equation*}
 		\bar{V}_i=\begin{pmatrix}
 		V_{i}(1) &  \cdots  &  V_i(mn(1+m) - m +1)  \\
 		\vdots      &           &      \vdots      \\
 		V_{i}(m) &    \cdots  & V_i(mn(1+m))
 		\end{pmatrix}  
 		\end{equation*}  \;
 		Build a   block Hankel matrix $H(V_i)$ having $m(d+1)$ rows  starting from the columns of $\bar{V}_i$ 
%
 	}  \;
 		\nl Stack    the matrices $H(V_i)$  in a row 
 		\begin{equation}
 		\label{eq:matk}
 		\mathcal{K} = [H(V_{md}), \dots, H(V_{1})]
 		\end{equation}
 		and compute $u_{m}, \dots, u_{1}$, the  left singular vectors of $\mathcal{K}$ associated with its $m$ smallest singular values   \;
 		
 	}
 \end{algorithm2e}

The following theorem shows how the proposed algorithm works.
\begin{thm}
      If the matrix $\mathcal{K}$ \eqref{eq:matk} is rank deficient, the subspace method computes a common factor between the  data matrix polynomials. Otherwise, it computes an approximate common factor.  
\end{thm}
\begin{proof}
	We show the result about the computation of exact common factors only; the if statement follows from the possible presence of noise but the algorithm is exactly the same. 

   In the case $A(z)$ and $B(z)$ have a (right) common factor $C(z)$, the resultant $S_{\ell}(A, B)$ can be split as $S_{\ell} (\bar{A}, \bar{B}) \tau(C)$. Moreover, we know the resultant $S_{\ell}(A, B)$ has a non-trivial kernel (see Section \ref{sec:sylvres}) so we can write the following 
   SVD factorization 
   $$S_{\ell}(A, B) = \begin{pmatrix}
   U_r &  U_0 \end{pmatrix} \begin{pmatrix}
   \Sigma_r & 0 \\ 0 & 0
   \end{pmatrix} \begin{pmatrix}
   V_r^{\top} \\V_0^{\top}
   \end{pmatrix},$$
   where $U_r, \Sigma_r, V_r$ correspond to the non-zero singular values/vectors.  We notice that the rows of $\tau(C)$ and the rows of $V_r^{\top}$ span the same subspace. Then, because of the orthogonality between  $V_r$ and $V_0$, the following equality holds true
   \begin{equation}
   \label{eq:v0}
   \tau(C) V_0 = 0. 
   \end{equation}
   Equation \eqref{eq:v0} has a unique solution for the common factor $C$ (up to multiplication by unimodular matrices, see Definition \ref{def:unimodular}) \cite{LiuXu}.
   Equation \eqref{eq:v0} can be written as
  \begin{equation}
  \label{eq:taucv}
  \sum_{\substack{i = md \\  i=i-1}}^{1} \| \tau(C) V_i \|^2 = 0 \ \ \ \ V_i \in V_0.
  \end{equation}
  Exploiting the Toeplitz structure of $\tau(C)$ the equation \eqref{eq:taucv} can be written as
  $$
  \sum_{\substack{i = md \\  i=i-1}}^{1}  \| C H(V_i)  \|^2 = 0    \ \ \ \ V_i \in V_0,
  $$
  where $C$ is a matrix collecting the  coefficients of the common factor (with an abuse of notation we use the same letter $C$), while $H(V_i) $ is a mosaic  Hankel matrix built from the entries of the vector $V_i$.
   Hence the entries of the matrix $C$, i.e. the coefficients of the sought common factor, can be recovered from the left null space of the matrix \eqref{eq:matk}  $\mathcal{K} = [H(V_{md}), \dots, H(V_{1})]$.

\end{proof}

\begin{rmk}
\label{rmk:fromgcdtopol}
Given the matrices $A(z)$ and $B(z)$, the subspace method computes only a (approximate) common factor $C(z)$ but not the  polynomials $\hat{A}(z), \hat{B}(z)$ having $C(z)$ as common factor. To compute these polynomials we need to solve the least squares problem 
\begin{equation*}
\label{eq:lsproblem}
\min_{\hat{A}, \hat{B}} \| A - \hat{A} \|_2^2 + \| B - \hat{B} \|_2^2 = 
  \min_{\bar{A}, \bar{B}} \| A - C \bar{A} \|_2^2 + \| B - C \bar{B} \|_2^2
\end{equation*}
where $C$ is the common factor computed by the algorithm.
\end{rmk}

	\begin{rmk}
(Computational cost) The advantage of this subspace method is to be very fast and cheap. The main computational cost consists in two SVDs.
\end{rmk}

\begin{rmk}
	The proposed algorithm computes a (exact) common factor between $A(z)$ and $B(z)$ whenever it exists. If the data do not admit a common factor, the algorithm automatically computes an approximate common factor, but there are no differences from the computational point of view.  
	\end{rmk}

\section{Generalized ODE-based method for matrix polynomials}
\label{sec:ode}
The goal of  this section is to generalize the algorithm proposed in \cite{Fazzi2018} for scalar polynomials,  to the case of matrix polynomials.  Even if some of the results stated in this section may look small variations of the one proposed in \cite{Fazzi2018, Guglielmi2017}, we remark that there are no algorithms in the literature which solve the considered  problem. Moreover, by removing the assumption in Remark \ref{rmk:fullrank}, we can change the objective functional in order to compute   approximate common factors   whose leading coefficient is rank deficient.  
A further difference with respect to the case of scalar polynomials is the computational strategy in  the outer iteration.

\subsection{General aspects}
We describe first some useful tools and ideas to understand how the algorithm works.  
When we deal with coprimeness of  matrix polynomials, just as it happens for the scalar case, the Sylvester resultant  is a useful tool. We showed in Section \ref{sec:sylvres} that, replacing the scalar coefficients by matrices, we do not have anymore the equality  between the corank (the dimension of the kernel) of the resultant and the degree of the common factor between the polynomials, as it happens in the scalar case \cite{sylvester1853}. In order to solve this issue,  it can  be worth to work with the modified  resultant $S_{\ell}$ \eqref{eq:modifiedres}, since in this way we  preserve the equality in \eqref{kergeqeigs}.

We start with a full rank Sylvester matrix $S_{\ell}(A, B)$ and we want to perturb the coefficients of the polynomials (in a minimal way) so that the kernel of the associated  resultant $S_{\ell}(\hat{A}, \hat{B})$ has dimension $k=md$. This is done by iteratively adding a structured perturbation to the matrix $S_{\ell}(A, B)$ which minimizes the singular values of interest (the $k$ smallest singular values).  
The rank test on the Sylvester matrix is done by computing its SVD, and in particular it is well known that a matrix has corank $k$ if and only if it has $k$ zero singular values. Exploiting the fact that the singular values are ordered non negative real numbers, we can focus on minimizing only the  $k$-th singular value. 
In particular we write the perturbed matrix as $\hat{S}_{\ell} =S_{\ell}+\epsilon E$, where $\epsilon$ is a scalar measuring the norm of the perturbation, while 
 $E$ is a norm one matrix (w.r.t. the Frobenius norm) 
which identifies as $\eps E$ the  minimizer of  $\sigma_k$  over the ball of matrices whose norm is at most $\eps$. 
In this way we can move $E$ and $\epsilon$ independently, minimizing the $k$-th singular value at one step, and the norm of the perturbation at the other until $\sigma_k = 0$.

These  ideas give raise to the following $2$-levels algorithm : we iteratively consider a matrix of the form $S_{\ell} + \epsilon E$ and we update it on two different independent levels:
\begin{itemize}
\item at the inner level we fix the value of $\epsilon$, and we minimize the functional $\sigma_{k}$ by looking for the stationary points of a system of ODEs for the matrix $E$;
\item at the outer level, we move the value of $\epsilon$ in order to compute the best possible solution.
\end{itemize}

\begin{rmk}
From the numerical point of view the functional $\sigma_k$  does not vanish, but it  only  reaches a fixed small  tolerance. 
\end{rmk}

\subsection{Inner iteration}
We analyze now the inner iteration of the algorithm, where the value of $\epsilon$ is fixed. The goal is to compute an optimal perturbation $E$  that  minimizes the singular value $\sigma_{k}$ of the matrix $S_{\ell}+\epsilon E$ over the set of matrices $E$ of unit Frobenius norm.
To do this we consider a smooth path of matrices $E(t)$ 
of unit Frobenius norm along which the singular value
$\sigma_k$ of $S_{\ell} + \eps E(t)$ decreases.
We exploit the following result about derivatives of eigenvalues for symmetric matrices \cite{Kato1976}.

\begin{lem}
\label{lemma:deriveig}
Let $D(t)$ be a differentiable real symmetric matrix function for $t$ in a neighborhood of $0$, and let $\lambda(t)$ be an eigenvalue of $D(t)$ converging to a simple eigenvalue $\lambda_0$ of $D(0)$ as $t \rightarrow 0$. Let $x_0$ be a normalized eigenvector (s.t. $x_0^{\top} x_0 = 1$) of $D_0$ associated to $\lambda_0$. Then the function $\lambda(t)$ is differentiable near $t=0$ with
\begin{equation*}
\dot{\lambda} = x_0^{\top} \dot{D} x_0
\end{equation*}
\end{lem}

Assuming that $E(t)$ is smooth we can apply Lemma \ref{lemma:deriveig} to the eigenvalues of the matrix $\hat{S}_{\ell}^{\top}(t) \hat{S}_{\ell}(t) = (S_{\ell} + \epsilon E(t))^{\top} (S_{\ell}+ \epsilon E(t))$, and we observe that the eigenvalues of $\hat{S}_{\ell}^{\top} \hat{S}_{\ell}$ are the squares of the singular values of $\hat{S}_{\ell}$ (we can assume the singular values are differentiable functions since, from the numerical point of view, we do not observe any coalescence among them). 
Omitting the time dependence, we find the following expression for the derivative of $\sigma_{k}$:
\begin{equation}
\label{eq:derivsigma}
\begin{aligned}
\dv{t} \sigma^2 &= v^{\top} \dv{t}(\hat{S}_{\ell}^{\top} \hat{S}_{\ell}) v = 2 \eps \sigma  u^{\top} \dot{E} v\\
\dot{\sigma}_{k} &= \epsilon u^{\top} \dot{E} v,
\end{aligned}
\end{equation}
where $u, v$ are the singular vectors of $\hat{S}_{\ell}$ associated to $\sigma_{k}$;
so the steepest descent direction for the functional $\sigma_k$,
minimizing the function over the admissible set for $\dot{E}$,
 is attained by minimizing $u^{\top} \dot{E} v = \langle uv^{\top}, \dot{E} \rangle$.
We notice that $E \in \mathcal{S}$, and consequently $\dot{E} \in \mathcal{S}$, hence
\begin{equation*}
\label{eq:equalinnerprod}
\langle uv^{\top}, \dot{E} \rangle = \langle P_{\mathcal{S}}(uv^{\top}), \dot{E} \rangle 
\end{equation*}
where the formula for the operator  $P_{\mathcal{S}}$ (the projection of the argument onto the Sylvester structure) is given in the following lemma:
\begin{lem}
Let $\mathcal{S}$ be the set of generalized Sylvester matrices of dimension $m\ell \times 2m(\ell - n)$, and let $H \in \R^{m\ell \times 2m(\ell-n)}$ be an arbitrary matrix. The orthogonal projection with respect to the Frobenius norm of $H$ onto $\mathcal{S}$ is given by (using Matlab notation for the rows/columns of the matrices)
$$
P_{\mathcal{S}}(H) = S_{\ell}(P^1, P^2),
$$
where 
$$
\begin{aligned}
P^1_{n-i} &= \frac{1}{\ell-n} \sum_{j=1}^{\ell-n} H(m(j-1)+1 : mj, m(j-1)+1+mi : m(j+i))  \\
P^2_{n-i} &= \frac{1}{\ell-n} \sum_{j=1}^{\ell-n} H(m(\ell-n) + m(j-1)+1 : m(\ell-n) + mj , \dots \\
                & \ \ \ \ \ \ \ \ \ \ \ \ \ \ \ \ \ \ \ \ \ \ \ \ \ \ m(j-1)+1+mi : m(j+i))  \\
                & \text{for} \ \ \   i = 0, \dots, n.
\end{aligned}
$$
\end{lem}
\begin{proof}
 The considered structured Sylvester matrices form a linear subspace and the basis matrices are orthogonal, the closest Sylvester matrix to a given matrix (in the Frobenius norm) is obtained by the inner product with the basis matrices (or equivalently taking the average along the diagonals).
\end{proof}
  
We underline that the projection $P_\mathcal{S}(uv^{\top})$ is different from zero for any pair of singular vectors $u, v$ associated to a non-zero singular value.
\begin{lem}
	\label{lemma:projnonzero}
	If $\sigma > 0$ is a simple singular value of a matrix  $\hat{S}_{\ell}$ with associated singular vectors $u$ and $v$, we have
	$$
	P_{\mathcal{S}} (uv^{\top}) \neq 0.
	$$
\end{lem}
	\begin{proof}
	 Assume, by contradiction, that we have $P_{\mathcal{S}} (uv^{\top}) = 0$.
	Doing some computations, we get
	\begin{equation}
	\label{eq:projnonzero}
	0 = \langle  P_{\mathcal{S}}(uv^{\top}), \hat{S}_{\ell} \rangle  =  \langle uv^{\top}, \hat{S}_{\ell} \rangle = u^{\top} \hat{S}_{\ell} v = \sigma > 0
	\end{equation}
	since $\sigma > 0$ by assumption. Consequently \eqref{eq:projnonzero}  is a contradiction, and the claim follows.
\end{proof}

\subsubsection{Minimization problem}
We found in \eqref{eq:derivsigma} the expression for the derivative of the singular value $\sigma_{k}$ of the Sylvester matrix $\hat{S}_{\ell}=S_{\ell} + \epsilon E$. In order to compute the steepest descent direction for  $\sigma_{k}$  we need to compute

\begin{equation}
\label{eq:minprob}
G = \textrm{arg}\min_{\substack{\dot{E} \in \mathcal{S} \\ \| \dot{E} \| = 1 \\ \langle E, \dot{E} \rangle = 1}} u^{\top} \dot{E} v 
\end{equation}
where the constraint on the norm is added in order to select a unique solution, since we look for a direction.
The solution of \eqref{eq:minprob} is given by:
\begin{equation}
\label{eq:ode}
\dot{E} = -P_{\mathcal{S}}(uv^{\top}) + \langle E, P_{\mathcal{S}}(uv^{\top}) \rangle E
\end{equation}
(the proof \cite[Section 4.2]{Guglielmi2017}  is based on the projection of an element in an Euclidean space onto the intersection of two linear subspaces). Consequently \eqref{eq:ode} is the key point of the inner iteration of the proposed algorithm. The following result shows its importance:

\begin{thm}
\label{th:derivsigmadecreasing}
Let $E(t) \in \mathcal{S}$ be a matrix of unit Frobenius norm, which is a solution of \eqref{eq:ode}. If $\sigma$ is the singular value of $\hat{S}_{\ell} = S_{\ell} + \epsilon E$ associated to the singular vectors $u, v$, then $\sigma(t)$ is decreasing, i.e.
\begin{equation*}
\dot{\sigma} \leq 0.
\end{equation*}
\end{thm}
\begin{proof}
In the proof we show that  $\dot{\sigma} \leq 0$. We remember the expression for $\dot{\sigma} = u^{\top} \dot{E} v$ (up to constant factors).  Exploiting  \eqref{eq:ode} to replace $\dot{E}$, we have two terms: the first is 
$$
u^{\top} P_{\mathcal{S}}(uv^{\top}) v = \langle uv^{\top}, P_{\mathcal{S}}(uv^{\top}) \rangle = \| P_{\mathcal{S}}(uv^{\top}) \|_F^2
$$
which follows from the structure of $P_{\mathcal{S}}(uv^{\top})$. The second is 
$$
u^{\top} \langle E, P_{\mathcal{S}}(uv^{\top}) \rangle E v = \langle E, P_{\mathcal{S}}(uv^{\top}) \rangle \langle E, uv^{\top} \rangle =  \langle E, P_{\mathcal{S}}(uv^{\top}) \rangle^2
$$
which follows from the Sylvester structure of $E$. 
Summing the two terms with the correct signs we have
$$
\dot{\sigma} = u^{\top} \dot{E} v = -\| P_{\mathcal{S}}(uv^{\top}) \|_F^2  +  \langle E, P_{\mathcal{S}}(uv^{\top}) \rangle^2 \leq 0
$$
since $\| E \|_F = 1$.
\end{proof}

\textcolor{blue}{
Theorem \ref{th:derivsigmadecreasing} and Lemma \ref{lemma:projnonzero}
guarantee that the function $\sigma(t)$ is monotonically decreasing (for a fixed value of $\epsilon$). Therefore  the stationary points of the ODE (corresponding to the zeros of $\dot{\sigma}$)  are the  candidate local minima for the functional under the considered constraints.} The following corollary provides a rigorous characterization
of minimizers.
\begin{corollary}
Consider a solution of  \eqref{eq:ode}, and assume the corresponding singular value $\sigma > 0$. The following statements are equivalent:
\begin{enumerate}
\item $\dot{\sigma} = 0$
\item $\dot{E} = 0$
\item $E$ is a scalar multiple of $P_{\mathcal{S}}(uv^{\top})$.
\end{enumerate}
\end{corollary}

\subsubsection{ODE integration}
We discuss here how to compute the solution of the ODE \eqref{eq:ode}.
Since \eqref{eq:ode} is a constrained gradient system, the value of $\sigma_k$ is monotonically decreasing, as we can see  in Figure \ref{fig:plotzvec}.
\begin{figure}[h]
\centering
\includegraphics[height=9cm, width=12cm]{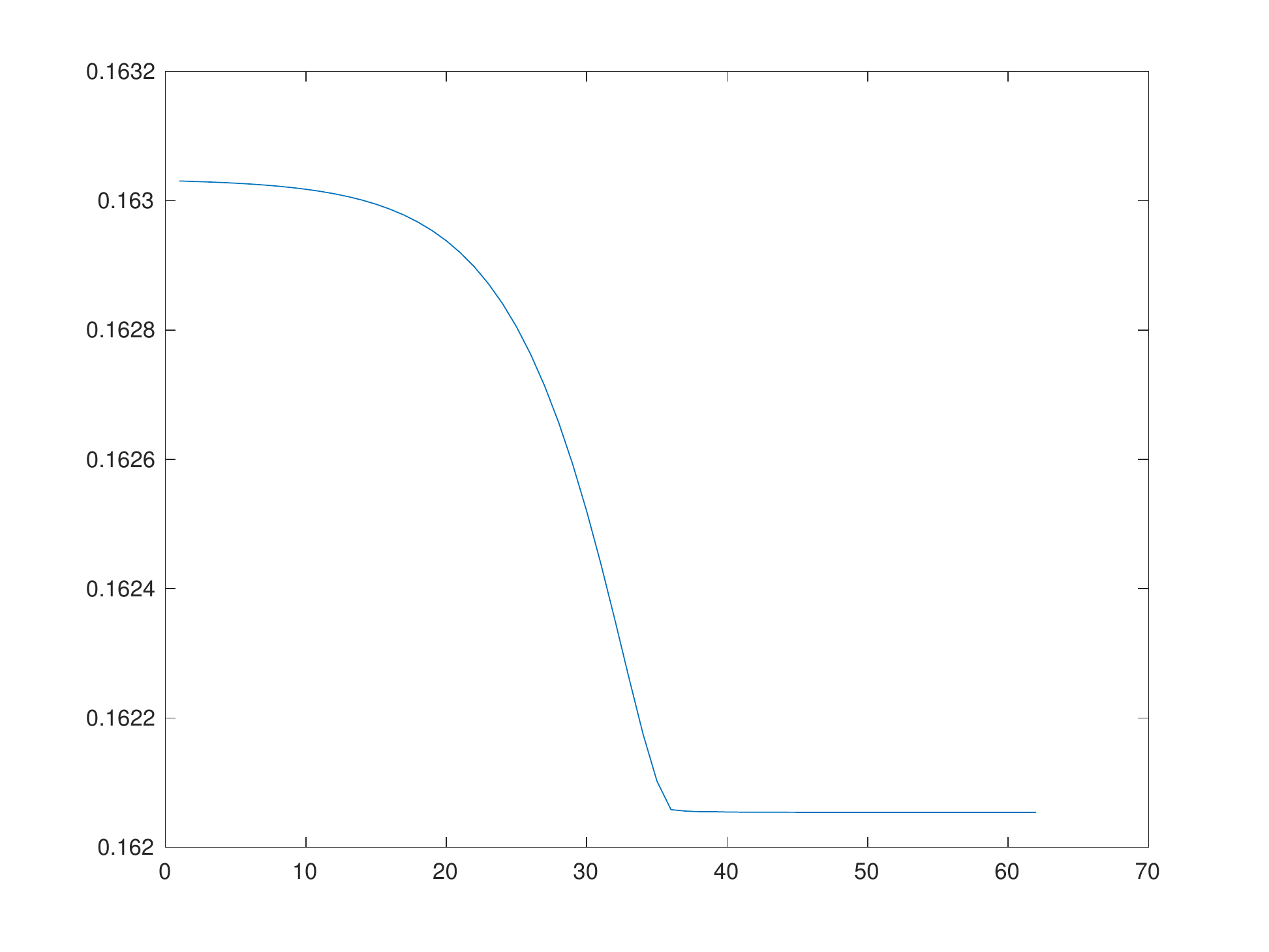}
\caption{\textcolor{blue}{Typical singular values $\sigma_k$ of   matrix $S_{\ell} + \epsilon E$ by increasing  iterations in the ODE  integration process  \eqref{eq:ode} (for a fixed $\epsilon$).  The final value corresponds to the stationary point of the equation \eqref{eq:ode}.The data for the plot are choosen randomly.}  \label{fig:plotzvec}}
\end{figure}
The function evaluation required in the integration of the equation is expensive because it involves the computation of a SVD  at each step (we need both the singular value and the corresponding singular vectors), so a suitable  choice is that of using  the explicit Euler scheme. We summarize the pseudo-code in Algorithm \ref{alg:controlODE}. We remark that the performances of the code can change depending on the values of some parameters (which depend on the starting data).

\begin{algorithm2e}[ht]
\caption{Numerical solution of the ODE \eqref{eq:ode}}
\label{alg:controlODE}
\DontPrintSemicolon
\KwData{$A, B$ (or equivalently the associated Sylvester matrix), $\sigma_k$, $u$, $v$, $h$ (step Euler method), $\gamma$ (step size reduction), \rm{tol}  and $\epsilon$.}
\KwResult{$\bar{E}$, $\bar{\sigma}_k$, $\bar{u}$, $\bar{v}$ and $\bar{h}$}

\Begin{
\nl Set $\tilde{h}=h$ \;
\nl Compute $\dot{E} = - P_S(u v^{\top}) + \langle E, P_S(u v^{\top}) \rangle E $ \;
\nl Euler step $\rightarrow$ $ \tilde{E} = E + h \dot{E}$ \;
\nl Normalize $\tilde{E}$ dividing it by its Frobenius norm \;
\nl Compute the  singular value $\tilde{\sigma}_k$ of the matrix $\tilde{S} = S_{\ell} + \epsilon \tilde{E}$ \;
\nl Compute the  singular vectors $\tilde{u}$ and $\tilde{v}$ of the matrix $\tilde{S}$ associated to $\tilde{\sigma}_k$ \;
\nl \eIf{ $\tilde{\sigma}_k > \sigma_k$}{
      reject the result and reduce the step $\tilde{h}$ by a factor $\gamma$ \;
      repeat from 3} 
      {accept the result; set $\bar{h} = h$, $\bar{\sigma}_k = \tilde{\sigma}_k$, $\bar{u} = \tilde{u}$, $\bar{v} = \tilde{v}$ \;
 \nl \If{$\bar{\sigma}_k - \sigma_k < tol$ \textbf{or} $\bar{\sigma}_k \leq tol$}
 {\textbf{return}} \;
    } 
 \nl \eIf{$\bar{h} = h$}
 {increase the step size of $\gamma$, $\bar{h} = \gamma h$}
 {set $\bar{h} = h$} \;
    
 }
\end{algorithm2e}

\subsection{Outer iteration}
\textcolor{blue}{
Once we integrate the ODE \eqref{eq:ode}, we find its stationary point $E$ and the corresponding $\sigma_k$. Since these quantities depend on a (fixed) value of $\epsilon$  we denote them as $\sigma_k(\epsilon), E(\epsilon)$.  The next step is to  find the minimal value $\eps$ (the norm of the 
perturbation to the original Sylvester matrix) which solves
the problem $\sigma_k(\eps) = 0$.
Observe that  the distance between the two matrices  is given by $\eps$ because of the relation $\hat{S}_{\ell} - S_{\ell} = \eps E$.}
  
Increasing the value of $\eps$, due to the  choice of an initial value
for the matrix $E$ in the gradient system \eqref{eq:ode}, can lead to unexpected    trajectories for the function $\sigma_k(\eps)$, that is
$\sigma_k(\eps)$ does not decrease.
The observed behavior can be due to possibly poor initialization for the ODE: it can happen that by  increasing the value of $\eps$ without changing the perturbation $E(\eps)$ in the initial datum, the equation reaches a stationary point before the objective functional decreases. 
In order to have a global decreasing
property with respect to both the inner and the outer iteration
we can 
 iteratively  alternate the following  dynamics:

\begin{enumerate}
	\item starting from the matrix $S_{\ell} + \hat{\epsilon} \hat{E}$, we integrate  (for a given $\eps > \hat{\eps}$)
	  the equation
	\begin{equation}
	\label{eq:odefree}
	\dot{E} = -P_S ( u v^T),  \qquad E(0) = \hat{E}
	\end{equation}
	where $\hat{E}$ is the computed equilibrium of the ODE \eqref{eq:ode}
	 corresponding to the value $\hat{\eps}$ and $u, v$ are the singular 
	vectors  associated with $\sigma_k$.

	 This equation is still a gradient system for the objective functional obtained from \eqref{eq:ode} by removing the constraints on the norm of $E$.  The solution is expected to increase in norm  while the objective functional decreases, so  we stop the integration of the equation  when the norm of the perturbation $E$  reaches the level
	 \begin{equation*}
	  \| E \|_F = \frac{\epsilon}{\hat{\eps}};
	  	\end{equation*}

	\item starting from the solution computed in point 1 (applying a normalization $\| E \|_F = 1$), integrate  \eqref{eq:ode} with initial datum $\eps E$  (using Algorithm \ref{alg:controlODE}). 
\end{enumerate}
The idea behind this computational strategy is to start each iteration at the endpoint of the  previous one, in a way that $\sigma_k(\epsilon)$ is continuous and monotonically decreasing with respect to $\epsilon$. This is obtained by integrating the ODE \eqref{eq:odefree} between two consecutive values of $\epsilon$.  The main body of this computational method is in Algorithm \ref{alg:ODE}.

\begin{algorithm2e}[H]
	\caption{ODE-based method for the computation of approximate common factors of matrix polynomials}
	\label{alg:ODE}
	\DontPrintSemicolon
	\KwData{$A, B$ ($m \times m$ matrix polynomials of degree $n$), $d$ (degree common factor), $tol$ (\textit{zero} tolerance), $\Delta$ (increase for the norm of perturbation)}
	\KwResult{$\hat{A}, \hat{B}$}

	\Begin{
		\nl Build $\mathcal{S} = \mathcal{S}_{\ell}(A, B)$ \;
		\nl Set $\epsilon=10^{-2}$           \ \ \ \ \ \ \ \ \ \ \ \ \      \%  starting value \;
		\nl Integrate the equation \eqref{eq:ode} \;
		store $E, \sigma=\sigma_{k}(\mathcal{S}+\epsilon E)$  \;
		\nl \While{$\sigma > tol$}
		{\nl $\epsilon_1=\epsilon + \Delta$  \;
			\nl       integrate the equation \eqref{eq:odefree} with initial datum $\mathcal{S} + \epsilon  E$ \;
			and stop when $ \epsilon / \epsilon_1 \| E \|_F \geq 1$ \;
			store $E, \sigma = \sigma_k(\mathcal{S}+\epsilon E)$  \;
			set $\epsilon_1= \epsilon  \| E \|_F$   \;
			\nl  integrate equation \eqref{eq:ode} with initial datum $\mathcal{S} + \epsilon_1 E$  \;
			store $E, \sigma = \sigma_k(\mathcal{S} + \epsilon_1 E)$    \;
			set $\epsilon=\epsilon_1$  }
		 $\hat{A}, \hat{B}$ are  recovered from $\mathcal{S}_{\ell}(\hat{A}, \hat{B}) = \mathcal{S} + \epsilon E$

	}
\end{algorithm2e}

\begin{rmk}
	(Computational cost) First of all we remark that the update of $\epsilon$ does not affect the computational cost since it is only one flop per iteration, and the two iteration levels (inner and outer) are independent. All the computations are developed at the inner level, i.e., during the integration of the gradient system. As described in this section, there are two different (alternating) dynamics: the unconstrained dynamic \eqref{eq:odefree} and the constrained one \eqref{eq:ode}.  The integration of each equation is an iterative algorithm which performs a SVD per iteration till the stopping criterion is reached (see Algorithm \ref{alg:controlODE}). Such decomposition is computed through the whole factorization of the matrix, hence the number of flops is expected to be cubic in the dimension of the data matrix (a possible improvement is object of future work). However it is not easy to estimate a priori the number of iterations needed by the integrator in order to reach the convergence, hence to guess the computational cost of the  algorithm. 
	As stated, the two iterations (inner and outer) are independent: however a poor accuracy in the inner iteration can determine also an
	inaccurate change of $\epsilon$, therefore a slowdown  of the process.
\end{rmk}

\subsection{GCD Computation}
\label{sec:gcdcomp}
In this paragraph we discuss how to extract the GCD from the perturbed polynomials computed by the ODE-based algorithm proposed in Section \ref{sec:ode}. We saw in Remark \ref{rmk:fromgcdtopol}  that given the GCD, we can obtain the polynomials $\hat{A}$, $\hat{B}$ by solving a least squares problem, but here the problem is more difficult. 

The first idea to compute the sought common factor from the non-coprime polynomials $\hat{A}, \hat{B}$ is to apply a fast and computationally cheap algorithm  (e.g. the subspace method proposed in Section \ref{sec:ss}). 

Alternatively we can make use of an external function for (exact) GCD computation for matrix polynomials.
 A suitable function comes from the Polyx Toolbox (\url{www.polyx.com}),  referring to the function \textit{grd.m} or \textit{gld.m} depending on the interest in computing a right or a left common factor, respectively. 
 
 \paragraph{\textbf{The functions \textit{grd.m} and \textit{gld.m}}}
 We briefly explain here how the two functions \textit{grd.m} and \textit{gld.m} from the Polyx Toolbox (\url{www.polyx.com}) work. We state the idea of the algorithm for right common factors computation, but dealing with left common factors has analogous counterparts. 
 
 Consider the matrix polynomials
 $$
 \begin{aligned}
 N_1(z) &= N_{10} + N_{11}z + \cdots + N_{1w}z^w \\
 N_2(z) &= N_{20} + N_{21}z + \cdots + N_{2w}z^w
 \end{aligned}
 $$ 
 having the same number of columns $m_N$, and define $N=\begin{bmatrix}
 N_1 \\ N_2
 \end{bmatrix} $. Consider the resultants $S_{w+ \ell}(N_1, N_2)$ (as defined in \eqref{eq:modifiedres}) for increasing $\ell=1, 2, \dots$. Each Sylvester matrix is then reduced to its shifted row Echelon form by a Gaussian elimination algorithm without row permutations. According to \cite{Barnett1983} the last $m_N$ nonzero rows of $S_{\bar{\ell}}$ yield the coefficients of a greatest common right divisor of $N_1, N_2$, where $\bar{\ell}$ is defined as the smallest integer such that 
 $$
 \rank (S_{w + \bar{\ell} +1}) - \rank (S_{w + \bar{\ell}}) = m_N.
 $$

However we remember these functions are thought for exact GCD computation, while the output polynomials computed by the proposed ODE based algorithm have not an exact GCD (the singular values of the resultant decrease up to a small tolerance but they do not reach the zero).
 In particular, we can observe some of the following issues:
\begin{enumerate}
\item the computed GCD equals the identity (so the functions are not able to reveal the presence of a common factor);
\item the leading coefficient of the GCD is singular, while we always assume the common factors  are monic (in particular the leading coefficient is full rank);
\item the computed GCD has degree higher than expected.
\end{enumerate}
Most of the times  no one of the previous facts is verified, and in these cases the common factors computed by the function \textit{grd.m} match the ones computed by the subspace method.

\section{Numerical experiments}
\label{sec:numericaltest}
\textcolor{blue}{In this section, we consider the performances of the proposed algorithms \ref{alg:sub} and \ref{alg:ODE}.} As stated before, there is no term of comparison in the scientific literature (up to our  knowledge), so the results of our algorithms are compared with the solutions obtained through the Matlab function $fminsearch$. 
\textcolor{blue}{
First of all we show a numerical example which highlights how the two Algorithms  \ref{alg:sub} and \ref{alg:ODE} work.  We consider   the following $2 \times 2$ matrix polynomials of degree $2$
\begin{equation*}
\begin{aligned}
\hat{A}(\lambda) &=  \begin{pmatrix}
 \lambda+1 & -\lambda \\  -\lambda+3 & -1
\end{pmatrix} \begin{pmatrix}
\lambda+1 & -1 \\ 1 &  \lambda+1
\end{pmatrix} = \begin{pmatrix}
\lambda^2+\lambda+1 & -\lambda^2 -2\lambda-1 \\
-\lambda^2+2\lambda+2 & -4
\end{pmatrix},  \\
\hat{B}(\lambda) &=  \begin{pmatrix} 
1 & -\lambda-1 \\ 3\lambda-1 & -\lambda
\end{pmatrix} \begin{pmatrix}
\lambda+1 & -1 \\ 1 &  \lambda+1
\end{pmatrix} = \begin{pmatrix}
0 & -\lambda^2-2\lambda-2 \\
3\lambda^2+\lambda-1 & -\lambda^2-4\lambda+1
\end{pmatrix}.
\end{aligned}
\end{equation*}
We generate then the data $A(\lambda), B(\lambda)$ by perturbing  all the coefficients with  normally distributed random noise with zero mean and standard deviation $0.1$.  Starting from the noisy data, Algorithm \ref{alg:sub} computes the following polynomials and the associated common factor:
\begin{equation*}
\begin{aligned}
\hat{A}^1(\lambda) &= \begin{pmatrix}
0.965 \lambda^2 + 0.951 \lambda + 0.309 &
-0.852 \lambda^2 -2.097 \lambda -1.216 \\
-1.185 \lambda^2 +1.608 \lambda +2.020 &
-0.391 \lambda^2 +0.018 \lambda -4.245
\end{pmatrix} \\
\hat{B}^1(\lambda) &= \begin{pmatrix}
0.509 \lambda^2 +1.032 \lambda +0.621 &
-0.724 \lambda^2 -1.594 \lambda -1.737 \\
2.827 \lambda^2 +0.569 \lambda -1.009 &
-1.349 \lambda^2 -4.100 \lambda +1.019
\end{pmatrix} \\
C^1(\lambda) &= \begin{pmatrix}
\lambda+0.762 & 0.191 \\ 1.949 & \lambda-0.798
\end{pmatrix}.
\end{aligned}
\end{equation*}
Starting from the same data, Algorithm \ref{alg:ODE} computes the following numerical solution:
\begin{equation*}
\begin{aligned}
\hat{A}^3(\lambda) &= \begin{pmatrix}
1.158 \lambda^2 +0.928 \lambda + 0.891 &
-1.000 \lambda^2 -2.074 \lambda -0.995 \\
-1.008 \lambda^2 +2.022 \lambda +2.166 &
-0.044 \lambda^2 +0.104 \lambda -4.154 
\end{pmatrix} \\
\hat{B}^3(\lambda) &= \begin{pmatrix}
0.040 \lambda^2 +0.070 \lambda -0.022 &
-0.991 \lambda^2 -1.902 \lambda -2.129 \\
2.940 \lambda^2 +0.889\lambda -1.013 &
-1.056 \lambda^2 -4.026 \lambda +1.039
\end{pmatrix} \\
C^3(\lambda) &= \begin{pmatrix}
\lambda+1.051 & -1.088 \\ 1.177 & \lambda+0.779
\end{pmatrix}.
\end{aligned}
\end{equation*}
We observe that both the polynomials and the common factor computed by Algorithm \ref{alg:ODE} are closer to the noiseless data than the ones computed by Algorithm \ref{alg:sub}.
}
\textcolor{blue}{
The result of the previous experiment is quite general: this is observed by  running now more  examples with random data, where we neglect the numerical values. 
}
We  generate  data polynomials having an exact common factor, and we add normal distributed  perturbations multiplied by a constant (called noise level) in the interval $[0, 1]$  in order to analyze the solution computed by the different approaches.  We focus only on the values of the computed distances. 
In the following experiments we generate fifty perturbations (for a given value of standard deviation) and we plot the average  distance computed by the different algorithms.

In  Figure \ref{fig:distr1} we have two $2 \times 2$ matrix  polynomials of degree $3$ and we compute an approximate (monic) common factor of degree one.

\begin{figure}[H]
\centering
\includegraphics[height=9cm, width=12cm]{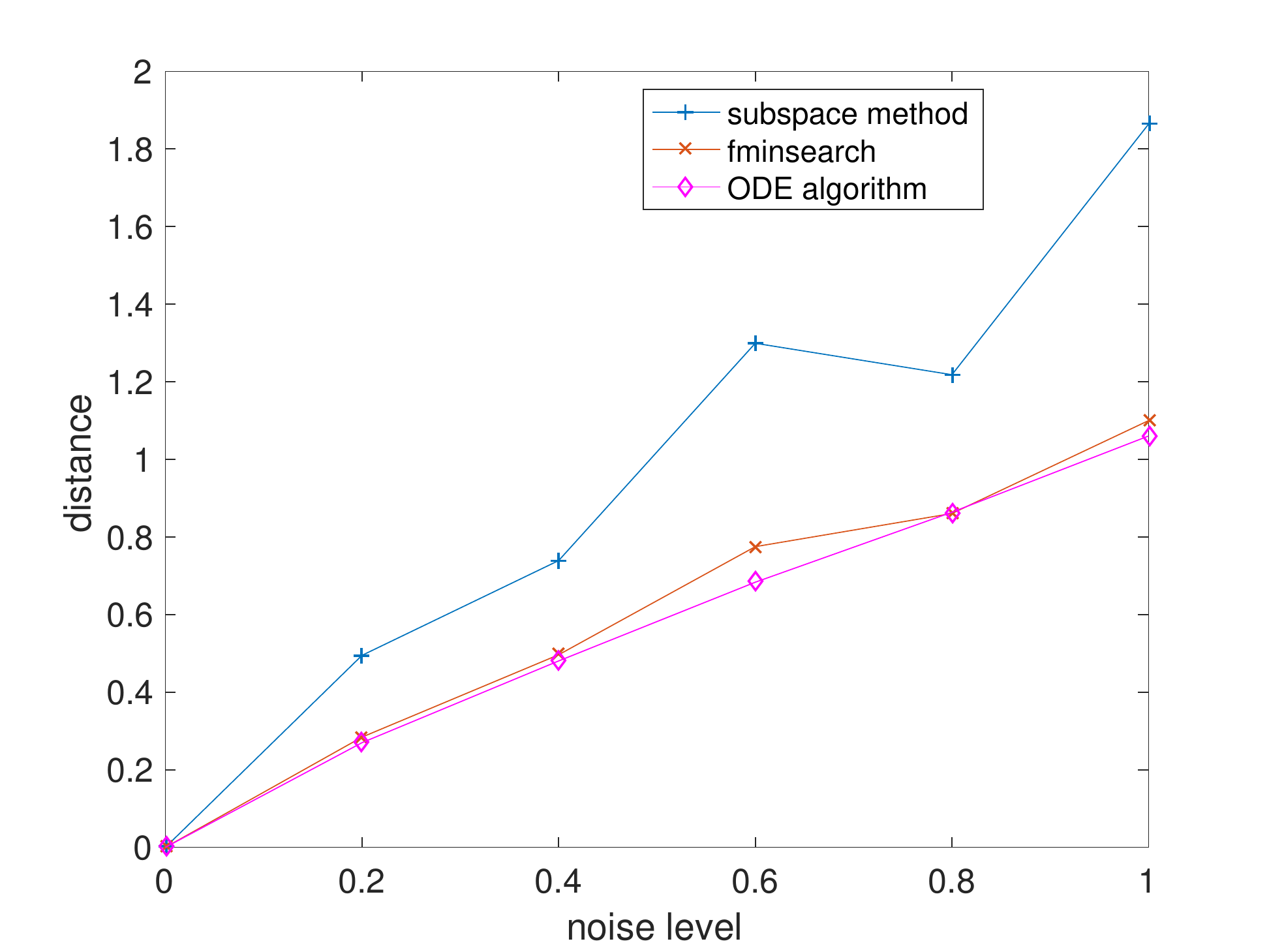}
\caption{\textcolor{blue}{Average distance as a function of noise intensity in an approximate common factor  of degree 1 for two $2 \times 2$ matrix polynomials of degree $3$.} \label{fig:distr1}}
\end{figure}

From the graph we can observe that the proposed ODE-based algorithm obtains better solutions (in terms of accuracy) than the subspace method, as it happened in the case of scalar polynomials \cite{Fazzi2018}.  
We need to make some comments about the minimization through the Matlab function \textit{fminsearch}. People familiar with Matlab know this function needs an initial approximation in input, so we can ask if the performances observed in Figure \ref{fig:distr1} depend on the (possibly poor) initialization. In Figure \ref{fig:distr1} the initial estimate is the solution computed by the subspace method, so it is not a bad choice but neither the best one since we observe the (average) computed distances are bigger than the ones computed by the ODE algorithm. If we initialize the function with the GCD computed by the proposed ODE-based algorithm, the solutions computed by the Matlab minimization improves the one got by the proposed method. In Figure \ref{fig:distr2} we observe a similar numerical example  where we added the distances computed by the function \textit{fminsearch} with different initializations (random, solution of the subspace method, solution of the ODE algorithm). We notice how the different initial estimates for the function \textit{fminsearch} influence the accuracy of the obtained solution. 

\begin{figure}[H]
\centering
\includegraphics[height=9cm, width=12cm]{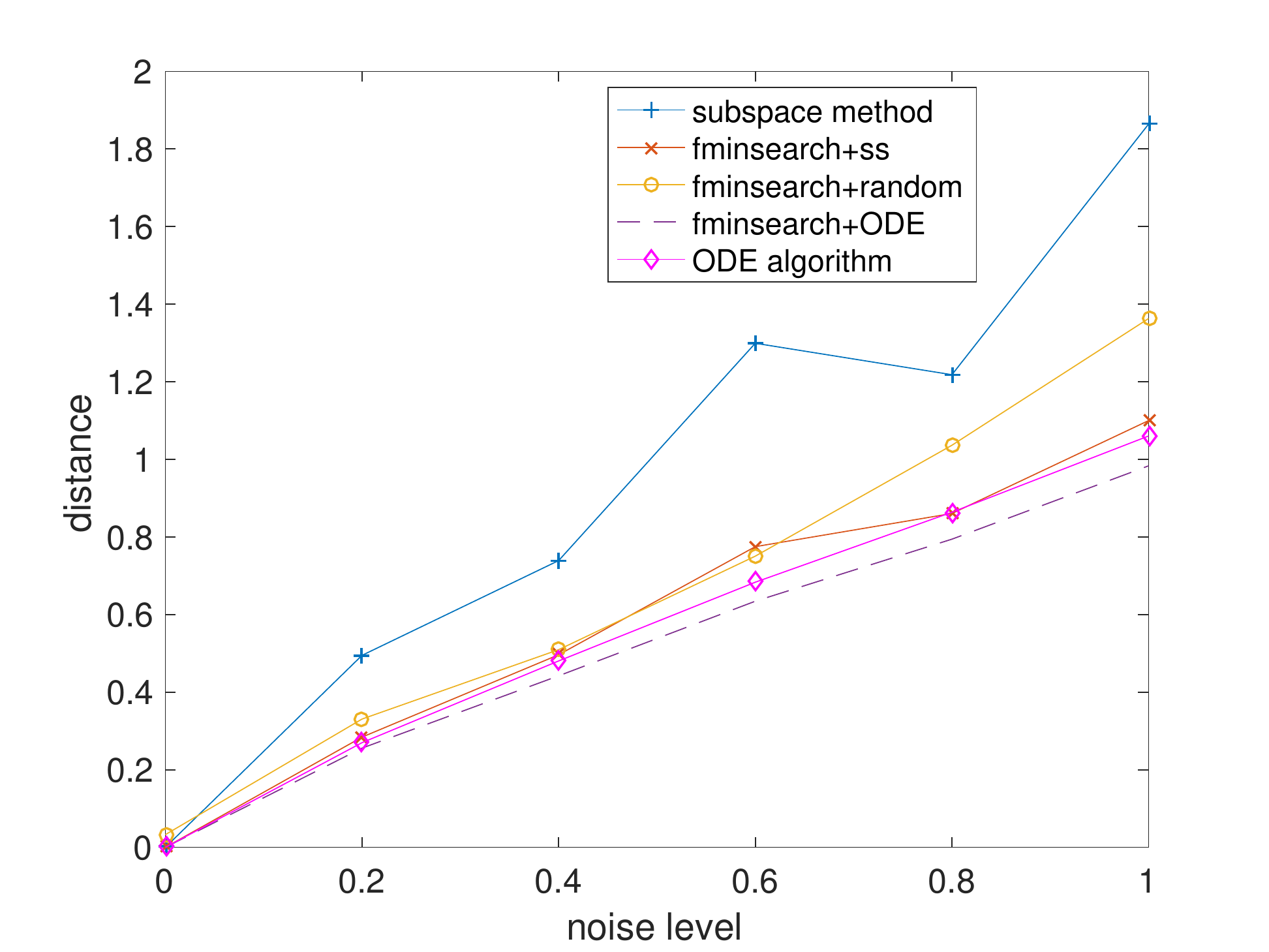}
\caption{\textcolor{blue}{Average distance as a function of noise intensity with different Matlab minimization processes in an  approximate common factor of degree 1 for two $2 \times 2$ matrix polynomials of degree $3$.} \label{fig:distr2}}
\end{figure}

\begin{rmk}
(Computational time) The subspace method is very fast
due to its low number of arithmetic operations. The proposed ODE-based 
algorithm  is (on average) faster than the function \textit{fminsearch}, whose 
performances depend on the initial estimate.
\end{rmk}

\section{Applications in system and control theory}
\label{sec:applications}
We show in this section an  application of the proposed algorithms. It extends the computation of \textit{distance to uncontrollability} from Single-Input Single-Output (SISO) systems (presented in \cite{Markovsky2018}) to Multi-Input Multi-Output (MIMO)  systems. However we  remind that any problem involving exact GCD computation for matrix polynomials can be seen as an approximate GCD computation problem whenever the coefficient are inexact, e.g. they come from measurements, computations \textcolor{blue}{or they are affected by perturbations \cite{Luo17}}.

\paragraph{Controllability for LTI systems}
Consider the linear time invariant system $\mathcal{B}$ defined by its state space representation
\begin{equation}
\label{eq:system}
\begin{cases}
\dot{x} &= Ax + Bu \\
y &= Cx + Du
\end{cases}
\end{equation}
where $A \in \R^{n \times n}$, $B \in \R^{n \times m}$, $C \in \R^{p \times n}$, $D \in \R^{p \times m}$.
The classical notion of controllability for \eqref{eq:system} is a property of the matrices $A, B$ and it is related to the rank of the matrix
\begin{equation}
\label{eq:contmat}
\mathcal{C}(A, B) = (B \ \   AB \  \  \cdots \ \   A^{n-1}B).
\end{equation}
In particular the system \eqref{eq:system} is state controllable if and only if the matrix $\mathcal{C}$ in \eqref{eq:contmat} is full rank. This definition of controllability is not a property of the system, but  of the matrices $A$ and $B$; consequently  distance problems associated to the matrix $\mathcal{C}$ may not have a well-defined solution since the same system \eqref{eq:system} can be represented by different parameters $(\hat{A}, \hat{B}, \hat{C}, \hat{D})$ (for example choosing a different basis or  considering a bigger state dimension). 

In order to avoid these issues we use the behavioral setting \cite{Polderman1998, Willems2007, Markovskybook2006}, where the notion of controllability is a property of the system and not of the parameters we choose for its representation. In this framework, the system \eqref{eq:system} is viewed as the set of its trajectories. The controllability property is the possibility of concatenating any two trajectories, up to a delay of time. 
\begin{definition}
\label{def:contr}
Let $\mathcal{B}$  be a time invariant dynamical system, which is  a set of trajectories (vector valued functions of time).   $\mathcal{B}$ is said to be controllable if for all $w_1, w_2 \in \mathcal{B}$ there exists a $T > 0$ and a $w \in \mathcal{B}$ such that
\begin{equation}
\notag
w(t)  = \begin{cases}  w_1(t)  &  \text{for} \ t < 0 \\
   w_2(t)  & \text{for}  \ t \geq T
\end{cases}
\end{equation}
A system is uncontrollable if it is not controllable. 
\end{definition}

Any linear time invariant system admit a kernel representation \cite{Willems1986};
hence given the system $\mathcal{B}$, there is a polynomial matrix $R(z) = (P(z) \ Q(z) ) \in \R^{p \times(m+p)}$ such that
\begin{equation}
\label{eq:kernelrep}
\mathcal{B}(R) = \{ w \  | \  R_0 w + R_1 \sigma w + \cdots + R_l \sigma^l w = 0 \}
\end{equation}
where $\sigma$ is the shift operator (in the discrete case). The controllability property is related to the rank of the matrix polynomial $R(z)$, and in particular we have the following Lemma \cite{Polderman1998}:
\begin{lem}
\label{lemma:rankR}
The system  $\mathcal{B}$ is controllable (according to Definition \ref{def:contr})  if and only if the polynomial matrix 
$$
R(z) = R_0 + R_1 z + \cdots + R_l z^l
$$
is left prime, i.e $R(z)$ is full row rank for all $z$.
\end{lem}

\paragraph{Distance to uncontrollability}
Alternatively to \eqref{eq:kernelrep}, a MIMO linear time invariant system can be represented by its input/output representation
$$
\mathcal{B}_{i/o} (P, Q) = \bigg\{ \begin{bmatrix} u \\ y \end{bmatrix}  \bigg| P(\sigma) y = Q(\sigma) u \bigg\}
$$
where we split the vector $w$ in \eqref{eq:kernelrep}  into two blocks (the inputs $u$ and the outputs $y$) and we partition the matrix $R=(Q \ \  -P)$ accordingly.
As a consequence of Lemma \ref{lemma:rankR} we have
\begin{corollary}\cite{kaashoekbook1989}
The presence of left common factors in $P$ and $Q$ leads to loss of controllability. 
\end{corollary}

Let $\mathcal{L}_{uc}$ be the set of uncontrollable linear time invariant systems with $m \geq 1$ inputs and $p \geq 1$ outputs,
$$
\mathcal{L}_{uc} = \{ \mathcal{B} \  | \  \mathcal{B} \ \text{uncontrollable} \   \text{MIMO} \  \text{LTI}  \  \text{system} \}
$$
and define the distance between two arbitrary systems by
$$
dist (\mathcal{B}(P, Q), \mathcal{B}(\bar{P}, \bar{Q})) = \| \begin{pmatrix}
P & Q
\end{pmatrix} - \begin{pmatrix}
\bar{P} & \bar{Q}
\end{pmatrix} \|_F,
$$
where the matrix  polynomials  are identified by a  vector whose entries are their coefficients\footnote{The parameters $P$ and $Q$ which identify the system are not unique. In order to have a well posed definition of distance we can assume $P$ to be monic. This involves however some loss of generality.}.
The problem of computing the distance to uncontrollability is the following:
\begin{problem}
\label{pr:duc}
Given a controllable system $\mathcal{B}(P, Q)$, find
$$
d(\mathcal{B}) = \min_{\mathcal{\bar{B}} \in \mathcal{L}_{uc}}  dist(\mathcal{B}, \mathcal{\bar{B}}).
$$
\end{problem}

In order to solve the non convex optimization Problem \ref{pr:duc}, we aim at perturbing the (left) coprime matrix polynomials  $P$ and $Q$ in a minimal way till they have a (left) common factor of degree $1$. The solution can be computed by the  algorithm proposed in Section \ref{sec:ode}.

A detailed description supported by some numerical experiments is presented in \cite{Fazzicdc}.

\section{Conclusions}
We generalized two algorithms for computing approximate common factors  from scalar to matrix polynomials. The first is a fast and computationally cheap algorithm which extract the informations about the common divisor from the resultant, while the second is a more accurate algorithm based on a two level iteration, which looks for the stationary points of a gradient system associated to a suitable functional. We showed how the performances are similar to the scalar case, and we described how to use the algorithms for computing the distance to uncontrollability for a Multi-Input Multi-Output linear time-invariant  system.

\paragraph*{\textbf{Acknowledgments}}
N. G. thanks the Italian
INdAM GNCS (Gruppo Nazionale di Calcolo Scientifico) for financial
support. I. M. received funding from the European Research Council (ERC) under 
the European Union's Seventh Framework Programme (FP7/2007--2013) / ERC 
Grant agreement number 258581 ''Structured low-rank approximation: 
Theory, algorithms, and applications"  and Fund for Scientific Research 
Vlaanderen (FWO) projects G028015N  ''Decoupling multivariate polynomials 
in nonlinear system identification" and  \\
G090117N  ''Block-oriented 
nonlinear identification using Volterra series"; and Fonds de la 
Recherche Scientifique (FNRS) -- FWO Vlaanderen under Excellence of 
Science (EOS) Project no 30468160  ''Structured low-rank matrix / tensor 
approximation: numerical optimization-based algorithms and applications". All the authors thank the anonymous reviewers  and the Principal Editor for their comments and suggestions, which led to an improvement of the paper.

\bibliographystyle{elsarticle-num}
\bibliography{bib}{}

\end{document}